\documentclass[reqno]{amsart}
\usepackage{mathrsfs}
\usepackage{amssymb}
\usepackage{fancyhdr}

\AtBeginDocument{}

\begin{document}
\title{On the optimal control problem for the k-FORQ/MCH equation with viscosity.}
\author[Z. Wang and S.Zhang]
{ Zhao Wang and Shan Zhang }

\address{Zhao Wang\hfill\break
School of Mathematics and Statistics, Changchun University of Technology, Changchun 130012,
China}
\email{wangzhao@ccut.edu.cn}

\address{Shan Zhang\hfill\break
	School of Mathematics and Statistics, Changchun University of Technology, Changchun 130012,
	China}
\email{zhangshan9039@126.com}

\thanks{}
\subjclass[2020]{49J20, 49K20}
\keywords{Optimal control, K-FORQ/MCH equation, Optimal solution,
Optimality condition.}

\begin{abstract}
In this paper, we investigate the optimal control problem for the k-FORQ/MCH equation with strong viscosity.We prove the existence and uniqueness of this equation under the initial and boundary conditions by Galerkin method. From these results and the Lion's theory we deduce the existence of an optimal solution to the control problem governed by the viscous k-FORQ/MCH equation. Using augmented Lagrangian method, we deduce the first-order necessary optimality condition and two second-order sufficient optimality conditions.

\end{abstract}

\maketitle
\let\oldsection\section
\renewcommand\section{\setcounter{equation}{0}\oldsection}
\renewcommand\thesection{\arabic{section}}
\renewcommand\theequation{\thesection.\arabic{equation}}
\newtheorem{theorem}{\indent Theorem}[section]
\newtheorem{lemma}{\indent Lemma}[section]
\newtheorem{proposition}{\indent Proposition}[section]
\newtheorem{definition}{\indent Definition}[section]
\newtheorem{remark}{\indent Remark}[section]
\newtheorem{corollary}{\indent Corollary}[section]
\newtheorem{claim}{\indent Claim}[section]

\section{Introduction}
This paper considered the following modified Camassa-Holm equation with cubic nonlinearity
\begin{align}
\label{1-1}
&y_t+(u^2-u_x^2)y_x+2u_xy^2+ku_x=0,  y=u-u_{xx}, k=constant.
\end{align}
which was introduced by Fokas \cite{A3}and Olver and Rosenau \cite{PP}, they employed the bi-Hamiltonian splitting method to the modified KdV equation. By using the approximation method , Qiao \cite{Z1} obtained the equation from the two-dimensional Euler equations. We call the k-FORQ/MCH equation for our convenience.

The FORQ/MCH equation is an representative member of the family of peakon equations. The peakon equations attracted much attention among the nonlinear systems and PDEs in recent years. This family contains the Camassa-Holm (CH) equation
\begin{equation*}
y_t+uy_x+2u_xy=0,  y=u-u_{xx}.
\end{equation*}
is a well-known model derived as a shallow water wave equation from the pioneering work of Camassa and Holm \cite{RD}. The equation was originally appeared in Fokas and Fuchssteiner in \cite{AB} and have been paid more attention in \cite{MJ, G, AJ5, AJ7, R2}, etc. The CH equation is the only peakon equation until the Degasperis-Procesi(DP) equation
\begin{equation*}
y_t+uy_x+3u_xy=0,  y=u-u_{xx}.
\end{equation*}
is introduced by the approach of asymptotic integrability \cite{DP, DHH} and possesses several nontrivial properties, such as multi-peakon solutions\cite{L, LS}, shock-type solutions\cite{L2},etc. Following some discoveries of peakon equations, a interesting model with quadratic and cubic nonlinearities is Novikov (N) equation
\begin{equation*}
y_t+u^2y_x+3uu_xy=0,  y=u-u_{xx}.
\end{equation*}
was first proposed by Novikov \cite{N}. The Novikov equation is a new integrable equation. Hone et al. \cite{HW, HLS} considered the  peakon and multi-peakon solutions, the stability and instability studied by Chen et al.\cite{YLZ, CHL, CPD}.

Nowadays, there are many researches on the k-FORQ/MCH equation. When $k=0$, the equation (\ref{1-1}) is called FORQ/MCH equation, Qiao \cite{Z1} disussed the integrability and structure of solutions. The Cauchy problem in Besov spaces and the blow-up phenomena of the solution was studied by Fu et al. \cite{YGYC}. Chen \cite{RYCS} and Liu \cite{YPCS} considered the blow-up phenomena of the solution. Qu et al.\cite{XYC,CXY} investigated the stability of single peakon and peakons. Bies et al. \cite{PPE}investigated the geometry and local analysis of the equation. Himonas et al. \cite{AD} considered the Cauchy problem. Whereas $k\neq 0$, Gui et al. \cite{GYP} studied the existence of peaked traveling-wave solutions, and the formation of singularities of solutions for this equation have also been considered. The k-FORQ/MCH equation was yielded the weak kink solutions \cite{Z2,BZ}, Qiao et al. \cite{Z3,BZ} discovered the peakon-kink interactional solutions which is a new phenomena shows that k-FORQ/MCH equation is not a simplely extension of FORQ/MCH.

On the other hand, there are a lot of research on the control problems of nonlinear PDEs in past decades. Volkwein \cite{SV} employed the augmented Lagrangian-SQP technique to study the optimal control problem of the Burgers equation. The optimal control of the incompressible Navier-Stokes flows in two and three dimensional space studied by Ghattas and Bark \cite{GB}. Sang-Uk Ryu and Atsushi Yagi \cite{RY} solved the optimal control problem governed by the  Keller-Segel equation. The optimal control of the Cahn-Hilliard equation have been considered \cite{MK,MHDW,Zhao}.

Besides, the multi-peakons equations had been paid more attention in academic field, a great deal of mathematical effort in optimal control problem has been devoted to the study of viscous multi-peakon equations, such as viscous Camassa-Holm equation\cite{Tian3}, viscous Degasperis-Procesi equation \cite{Tian2} and viscous Novikov equation \cite{zhou} in recent years. From now on, the optimal control of the FORQ/MCH equations have not been studied yet. The main purpose of this paper is to consider the optimal control problem governed by the viscous k-FORQ/MCH equation
\begin{align}
\label{1-4} y_t-\epsilon y_{xx}+(u^2-u_x^2)y_x+2u_xy^2+ku_x=0,y=u-u_{xx}, k=Constant.
\end{align}
where the parameter $\epsilon>0$ denotes strong viscosity. The key issue for this problem is to deal with the cubic nonliearity $u_xy^2$. We take advantage  of the structure of equation (\ref{1-1}) to overcome this difficulty.

This paper is organized as follows. In the next section, we introduce some functional spaces and the control problem governed by the k-FORQ/MCH equation. In section 3, We also discuss the relation among the norms of weak solution, initial value and control item; In section 4, we consider the optimal control problem of the equation (\ref{1-1}) and prove the existence of optimal solution; In the fifth section , we deduce the surjectivity of the operator $e'(y, \omega)$ which guarantees the first-order necessary optimality condition. In Section 6, We establish two second-order sufficient optimality conditions, which require coercivity of the augmented Lagrangian functional on the whole space or on a suitable subspace.

\section{Preliminaries}\label{sec2}

\subsection{Notations}

In the following, we introduce some notations that will be used
throughout the paper. For fixed $T>0$, $Q=\Omega\times(0,T)$ and $Q_0\subseteq Q$
be an open set with positive measure,
$$V=\left\{y\in H^1(\Omega);
y(x,t)=D y(x,t)=0,~x\in\partial\Omega\right\},$$ and $H=L^2(\Omega),$ let $V'$ and
$H'$ are dual spaces of $V$ and $H$. Then, we obtain
$$
V\hookrightarrow  H=H'\hookrightarrow  V'.
$$
Obviously, each embeddings being dense.

We supply $H$ with the inner product $(\cdot,\cdot)$ and the norm
$\|\cdot\|$, and define a space $W(0,T; V)$ as
$$
W(0,T;V)=\left\{y;y\in L^2(0, T;V),~\frac{dy}{dt}\in
L^2(0,T;V')\right\},
$$
which is a Hilbert space endowed with common inner product, the corresponding norm being
denoted $\|\varphi\|_{W(0, T; V)}$,
$$
\|\varphi\|_{W(0, T; V)}=(\|\varphi\|_{L^2(0, T; V)}+\|\varphi_t\|_{L^2(0, T; V^*)})
$$
For convenience we write $L^2(V), C(H)$ and $W(V)$ in place of  $L^2(0, T; V), C(0, T; H)$ and $W(0, T; V)$.
Since $W(V)$ is continuously embedded into $C(H)$, there exists an embedding constant $c_E>0$ such that
\begin{equation}
\label{2-1}
\|\varphi\|_{C(H)}\leq c_E\|\varphi\|_{W(V)}~~\hbox{for all}~~\varphi\in W(V)
\end{equation}

\subsection{Formation of the optimal control problem}
First of all, we introduce extension operator $B\in \mathcal {L}\left(L^2(Q_0),
L^2(0,T;V')\right)$ which is called the controller is introduced as
\begin{eqnarray}
Bq= \left\{\begin{array}{ll}q,&~ q\in Q_0,
\\0,&~ q\in Q\setminus Q_0.\end{array}\right.
\nonumber
\end{eqnarray}

In this paper, we are concerned with distributed optimal control
problem
\begin{equation}
\label{2-2}
\min
J(y,u)=\frac12\|Gy-z_d\|^2_S+\frac{\delta}{2}\|\omega\|^2_{L^2(Q_0)},
\end{equation}
subject to
\begin{align*}
\label{2-3}
\begin{cases}
y_t-\epsilon y_{xx}+(u^2-u_x^2)y_x+2u_xy^2+ku_x=B\omega,
\\
y(x,t)=D^2 y(x,t)=D^3 y(x,t)=0,~x\in\partial\Omega,
\\
y(x,0)=y_0.
\end{cases}
\end{align*}
where $y=u-u_{xx}$ and $\delta>0$ is fixed, $G\in \mathcal {L}(W(0,T;V), S)$ is an operator, which is
called the observer, $S$ is a real Hilbert space of observations.
The control target is to match the given desired state $z_d$ in $L^2$-sense by adjusting the body force $\omega$ in a control volume
$Q_0\subseteq Q=\Omega\times(0,T)$ in the $L^2$-sense.

Let $X=W(0,T;V)\times L^2(Q_0)$ and $Y=L^2(0,T;V)\times H$. We
define an operator $e=e(e_1,e_2):X\rightarrow Y$, where
$$
e_1=(-\Delta)^{-1}(y_t-\epsilon y_{xx}+(u^2-u_x^2)y_x+2u_xy^2+ku_x-B\omega),
$$
and
$$
e_2=y(x,0)-y_0.
$$
Here $\Delta$ is an operator from $V$ to $V'$. Hence, we write
(\ref{2-2}) in following form
\begin{equation}
\label{2-3}
\min J(y,u)~~\hbox{subject~to}~~e(y,u)=0.
\end{equation}
where $J:X\rightarrow\mathbb{R}$ and $e:X\rightarrow Y
$ are twice continuously Fr$\acute{e}$chet-differentiable and their second Fr$\acute{e}$chet-derivatives are Lipschitz continuous. In the following the Fr$\acute{e}$chet-derivatives with respect to the variable $(y,\omega)$ .

For $c\geq 0$, we introduce the functional $L_c: X\times Y\rightarrow\mathbb{R}$ associated with the control problem:
\begin{equation*}
L_c(y, \omega, \lambda, \mu)=J(y,\omega)+(e_1(y, \omega), \lambda)_{L^2(V)}+(e_2(y, \omega), \mu)_H+\frac{c}2\|e(y, \omega)\|_Y^2
\end{equation*}
where $J(y,\omega)$ is above chosen cost function and the Hilbert space $X\times Y$ is endowed with the Hilbert space product topology.

\section{Existence and uniqueness of weak solution}\label{sec2}

In this section, we prove the existence and uniqueness of weak
solution for the following equation
\begin{align}
\label{3-1} y_t-\epsilon y_{xx}+(u^2-u_x^2)y_x+2u_xy^2+ku_x=B\omega,y=u-u_{xx},~(x,t)\in Q,
\end{align}
with the boundary conditions
\begin{align}
\label{3-2} y(x,t)=D y(x,t)=D^2y(x,t)=0,~x\in\partial\Omega,
\end{align}
and the initial value condition
\begin{align}
\label{3-3} y(x,0)=y_0(x),~~x\in \Omega,
\end{align}
where  $B\omega\in L^2(0,T;H)$ and a control
$\omega\in L^2(Q_0)$.

We shall give the definition of weak solution in the space
$W(0,T;V)$ for the problem (\ref{3-1})-(\ref{3-3}).

\begin{definition}
\label{def3.1} For all $\eta\in V,~t\in(0,T)$, a function $y(x,t)\in
W(0,T;V)$ is called a weak solution to problem
(\ref{3-1})-(\ref{3-3}), if 
\begin{align*}
\frac{d}{dt}(y,\eta)_{V',V}+\epsilon( u, \eta)_V+((u^2-u_x^2)y_x+2u_xy^2+ku_x,\eta)_{V',V}=(B\omega,\eta)_{V',V}.
\end{align*}
\end{definition}

Now, we give Theorem \ref{thm3.1}, which ensures the existence of a
uniqueness weak solution to problem (\ref{3-1})-(\ref{3-3}).

\begin{theorem}
\label{thm3.1} Suppose $y_0\in H$,$B\omega\in L^2(0,T;V')$ and $\|y_0\|^2+C(\varepsilon)T\|B\omega\|^2<(e^{2C(\varepsilon)T}-1)^{-\frac12}$,  then
the problem (\ref{3-1})-(\ref{3-3}) admits a unique weak solution
$y(x,t)\in W(0,T;V)$. Moreover, there exists a constant $c>0$, which depends on $T$ and $\epsilon$ , such that
$$
\|y\|_{W(V)}\leq C(\exp\|y_0(x)\|^2_H+\|\omega\|^2_{L^2(Q_0)}+1)
$$
\end{theorem}
\begin{proof} Galerkin method is applied to the proof.

Denote $\mathbb{A}=-\Delta$ as a differential operator, let
$\{\psi_i\}_{i=1}^{\infty}$ denote the eigenfunctions of the
operator $\mathbb{A}=-\Delta$. For $n\in N$, define the
discrete ansatz space by
$$
V_n=\hbox{span}\{\psi_1,\psi_2,\cdots,\psi_n\}\subset V.
$$
Let $y_n(t)=y_n(x,t)=\displaystyle\sum_{i=1}^{n}y^n_i(t)\psi_i(x)$ require
$y_n(0,\cdot)\rightarrow y_0$ in $H$ holds true.

By analyzing the limiting behavior of sequences of smooth function
$\{y_n\}$, we can prove the existence of a weak solution to the
problem (\ref{3-1})-(\ref{3-3}).

Performing the Galerkin procedure for Eq.(\ref{3-1}), we obtain 
\begin{align}
 y_{n,t}-\epsilon y_{n,xx}+(u_n^2-u_{n,x}^2)y_{n,x}+2u_{n,x}y_{n}^2+ku_{n,x}=B\omega,~~(x,t)\in Q,
\label{3-4}
\end{align}
with the boundary conditions
\begin{align}
\label{3-5}y_n(x,t)=D y_n(x,t)=D^2 y_n(x,t)=0,~x\in\partial\Omega,
\end{align}
and the initial value condition
\begin{align}
\label{3-6} y_n(x,0)=y_{n,0}(x),~~x\in \Omega.
\end{align}

Obviously, Eq.(\ref{3-4}) is an ordinary differential equation and
according to ODE theory, there exists a unique solution to the
problem (\ref{3-4})-(\ref{3-6}) in the interval $[0,t_n)$. What we
should do is to show that the solution is uniformly bounded when
$t_n\rightarrow T$. We need also to show that times $t_n$
are not decaying to $0$ as $ n\rightarrow\infty$.

Then, we shall prove the existence of solution in the following
steps.

Step 1. Multiplying Eq.(\ref{3-4}) by $u_n$, integrating with
respect to $x$ on $\Omega$, then we obtain
\begin{align*}
&\frac12\frac d{dt}\int_\Omega (u_n^2+u_{n,x}^2)dx+\epsilon\int_\Omega (u_{n,x}^2+u_{n,xx}^2)dx
=(B\omega,u_n)_{V',V}
\end{align*}
By H\"{o}lder inequality and Young's inequality, we obtain
\begin{align*}
(B\omega,u_n)_{V',V}\leq\frac1{4\epsilon}\|B\omega\|_{V'}^2+\epsilon \|u_n\|_V^2.
\end{align*}
Then, we have
\begin{align*}
&\frac12\frac d{dt}\int_\Omega (u_n^2+u_{n,x}^2)dx+\epsilon\int_\Omega u_{n,xx}^2dx\leq\frac1{4\epsilon}\|B\omega\|_{V'}^2.
\nonumber
\end{align*}
Hence,  we have
\begin{align*}
\frac{d}{dt}\int_\Omega (u_n^2+u_{n,x}^2)dx+\epsilon\int_\Omega u_{n,xx}^2dx\leq C\|B\omega\|_{V'}^2.
\end{align*}
Since $B\omega\in L^2(0,T;V')$ is the control item, we obtain
\begin{align}
\label{3-7}
\int_\Omega (u_n^2+u_{n,x}^2)dx=\|u\|_V^2 \leq C,~\forall t\in[0,T],
\end{align}
and
\allowdisplaybreaks
\begin{align}
&\sup_{0<t<T}\int_\Omega u_{n,xx}^2dx\leq C,
\label{3-8}
\end{align}
where $C$ depends only on $u_n(0)+u_{n,x}(0)$.

On the other hand,  we have
$$
\|u\|_ \infty\leq \|u\|_V \leq C,\|u\|_H\leq C.
$$

Step 2.Multiplying Eq.(\ref{3-4}) by $y_n$, integrating with
respect to $x$ on $\Omega$, then we obtain
\begin{align}
&\frac12\frac d{dt}\int_\Omega y_n^2dx+\epsilon\int_\Omega y_{n,x}^2dx
\nonumber
\\
&=\int_\Omega (u^2_n-u^2_{n,x})y_ny_{n,x}dx+ku_{n,x}y_n+(B\omega,y_n)_{V',V}
\nonumber
\\
&=\int_\Omega u^2_ny_ny_{n,x}dx-\int_\Omega u^2_{n,x}y_ny_{n,x}dx+\int_\Omega ku_{n,x}y_ndx+(B\omega,y_n)_{V',V}
\nonumber
\\
&=\int_\Omega u^2_ny_ny_{n,x}dx+\int_\Omega u_{n,x}u_{n,xx}y^2_ndx+(B\omega,y_n)_{V',V}
\label{3-9}
\end{align}

Using H\"{o}lder's inequality and Young's inequality for products, we obtain
\begin{align}
\int_\Omega u^2_ny_ny_{n,x}&\leq\|u\|^2_\infty\int_\Omega y_ny_{n,x}dx
\nonumber
\\
&\leq C\int_\Omega y_ny_{n,x}dx
\nonumber
\\
&\leq  \frac3{2\epsilon}\int_\Omega y_n^2dx+\frac\epsilon6\int_\Omega y_{n,x}^2dx
\label{3-10}
\end{align}

Note that $y=u-u_{xx}$, then $u=(1-\partial_x)^{-1}y=p*y$ where $p(x)=\frac12e^{-|x|}$,$u_{xx}=u-y$ and $\|p\|_{L^1}=\|\partial_x p\|_{L^1}=1$ . By the Young's inequality, we have
\begin{align}
\|u\|_{L^\infty},\quad\|u_x\|_{L^\infty}\leq C\|y\|_{L^\infty} \qquad
\hbox{and}\qquad
\|u\|_{L^\infty}, \quad\|u_x\|_{L^\infty}\leq C\|y\|_H.
\label{3-16}
\end{align}

Using the Sobolev embedding theorem, we obtain
\begin{align}
\int_\Omega u_{n,x}u_{n,xx}y^2_ndx&\leq \|u_{n,x}\|_\infty \|u_{n,xx}\|_\infty\int_\Omega y^2_ndx
\nonumber
\\
&\leq C\|y_n\|_H \|y_n\|_V \int_\Omega y^2_ndx
\nonumber
\\
&\leq C\|y_n\|^3_H \|y_n\|_V
\nonumber
\\
&\leq C(\epsilon)\|y_n\|^6_H +\frac\epsilon6 \|y_n\|^2_V
\label{3-11}
\end{align}

By H\"{o}lder inequality and Young's inequality, we obtain
\begin{align}
(B\omega,u_n)_{V',V}\leq C(\epsilon)\|B\omega\|_{V'}^2+\frac\epsilon6 \|u_n\|_V^2.
\label{3-13}
\end{align}

Combining (\ref{3-9})-(\ref{3-13}), we see that
\begin{align}
& \frac d{dt}\|y_n\|^2_H+\epsilon\|y_n\|^2_V
\nonumber
\\
\leq& C(\epsilon)(\|B\omega\|_{V'}^2+\|y_n\|^6_H+\|y_n\|^2_H)
\end{align}

By the Gronwall's inequality, we derive that
\begin{align}
\|y_n\|^2_H\leq \dfrac{e^{C(\epsilon)t}A}{\sqrt{(1-e^{2C(\epsilon)t})A+1}}\leq C(T)
\label{3-17}
\end{align}

A simple calculation reveals that
\begin{align}
\int_\Omega y^2_ndx=&\int_\Omega (u_n-u_{n,xx})^2dx
\nonumber
\\
=&\int_\Omega (u^2_n-2u_nu_{n,xx}+u_{n,xx}^2)dx
\nonumber
\\
=&\int_\Omega u^2_ndx+2\int_\Omega u^2_{n,x}dx+\int_\Omega u_{n,xx}^2dx
\label{3-14}
\end{align}

Integrating (\ref{3-14}) with respect to t on $[0, T]$, we derive that
\begin{align}
&\int_0^T\|y_n\|_H^2dt
\nonumber
\\
=&\int_0^T\|u_n\|_H^2dt+2\int_0^T\|u_n\|_V^2dt+\int_0^T\int_\Omega u_{n,xx}^2dxdt
\nonumber
\\
\leq& C
\label{3-15}
\end{align}
Meanwhile, we derive from (\ref{3-14}) that
\begin{align}
\int_\Omega u_{n,xx}^2dx\leq C
\label{3-18}
\end{align}
and from (\ref{3-17}) and (\ref{3-15}) that
\begin{align}
\|y_n\|_{L^2(V)}&=\int_{0}^{T}\|y_n\|^2_Vdt
\nonumber
\\
&\leq C\|y_{n,0}(x)\|^2_H+\int_0^T\|u_n\|_H^2dt+C(\epsilon)\int_0^T\|y_n\|^6_Hdt\leq C
\label{3-19}
\end{align}

By H\"{o}lder's inequality, Poincar\'{e}'s inequality and the Sobolev embedding theorem, we deduce from the first equation in (\ref{3-4}) that
\begin{align}
&\|y_{n,t}\|_{L^2(V')}
\nonumber
\\
=&\underset{\|\varphi\|_{L^2(V)}=1}{\sup}\bigg|\int_0^T\int_\Omega y_{n,x}\varphi_x+\int_0^T\int_\Omega (u^2_n-u^2_{n,x})y_n\varphi_xdxdt
\nonumber
\\
&\qquad+k\int_0^T\int_\Omega u_{n,x}\varphi dxdt+\int_0^T(B\omega,\varphi_x)_{V',V}dt\bigg|
\nonumber
\\
=&\underset{\|\varphi\|_{L^2(V)}=1}{\sup}\bigg|\int_0^T\int_\Omega y_{n,x}\varphi_x+\int_0^T\int_\Omega u^2_ny_n\varphi_xdxdt
+\int_0^T\int_\Omega 2u_{n,x}u_{n,xx}y_n\varphi dxdt
\nonumber
\\
&\qquad+k\int_0^T\int_\Omega u_{n,x}\varphi dxdt+\int_0^T\int_\Omega u_{n,x}^2 y_{n,x}\varphi dxdt+\int_0^T(B\omega,\varphi)_{V',V}dt\bigg|
\nonumber
\\
\leq&\underset{\|\varphi\|_{L^2(V)}=1}{\sup}\bigg(\int_0^T\| y_{n,x}\|_H\|\varphi_x\|_Hdt
+\int_0^T\| u_n\|^2_\infty\| y_n\|_H\|\varphi_x\|_Hdt
\nonumber
\\
&\qquad\qquad+\int_0^T\|u_{n,x}\|_\infty \|\varphi\|_\infty\|u_{n,xx}\|_H\|y_n\|_Hdt+\int_0^T\|u_{n,x}\|_\infty \|\varphi\|_Hdt
\nonumber
\\
&\qquad\qquad+\int_0^T\|u_{n,x}\|^2_\infty \|y_{n,x}\|_H\|\varphi\|_Hdt+\int_0^T\|B\omega\|_{V'}\|\varphi\|_Vdt\bigg)
\nonumber
\\
\leq&C \|y_n\|_{L^2(V)}+\|B\omega\|_{L^2(V')}\leq C.
\label{3-20}
\end{align}
where we have used $\|\varphi\|_\infty\leq\|\varphi\|_V,\|u_{n,x}\|_\infty\leq\|u_{n,xx}\|_H\leq C.$

Thus, we have:

(1) For every $t\in [0,T]$, the sequence $\{y_n\}_{n\in N}$ is
bounded in $L^2(0,T;V)$, which is independent of the dimension of
ansatz space $n$.

(2) For every $t\in[0,T]$, the sequence $\{y_{n,t}\}_{n\in N}$ is
bounded in $L^2(0,T;V')$, which is independent of the dimension of
ansatz space $n$.

Based on the above discussion, we obtain $\{y_{n, t}\}_{n\in\mathbb{N}}\subset W(0,T;V)$ and  $W(0,T;V)$ is compactly embedded into
$C(0,T;H)$ which denotes the space of continuous functions. We
conclude convergence of a subsequences, again denoted by $\{y_n\}$
weak into $W(0,T;V)$, weak-star in $L^{\infty}(0,T;H)$ and strong
in $L^2(0,T;H)$ to functions $y(x,t) \in W(0,T;V)$.  Since the
proof of uniqueness is easy, we omit it.

We also deduce from (\ref{3-14}) and (\ref{3-20})   that
\begin{align*}
\|y_t\|_{W(V)}&=\|y\|_{L^2(V)}+\|y_t\|_{L^2(V')}
\nonumber
\\
&\leq C\|y_0(x)\|^2_H+C+C \|y\|^2_{L^2(V)}+\|B\omega\|_{L^2(V')}
\nonumber
\\
&\leq C(\|y_0(x)\|^2_H+\|\omega\|^2_{L^2(Q_0)}+1)
\end{align*}

Then, Theorem \ref{thm3.1} is proved.
 \end{proof}

\section{Optimal control problem}\label{sec4}

In this section, we consider the optimal control
problem associated with the equation (\ref{3-1}) and prove
 the existence of optimal solution basing on Lions' theory (see \cite{Lions}).

In the following, we suppose $L^2(Q_0)$ is a Hilbert space of
control variables, we also suppose $B\in \mathcal {L}(L^2(Q_0),
L^2(0,T;H))$ is the controller and a control $u\in L^2(Q_0)$,
consider the following control system
\begin{align}
\begin{cases}
\label{4-1} y_t-\epsilon y_{xx}+(u^2-u_x^2)y_x+2u_xy^2+ku_x=B\omega,
\\
y(x,t)=D^2 y(x,t)=D^3 y(x,t)=0,~x\in\partial\Omega,
\\
y(x,0)=y_0.
\end{cases}
\end{align}
Here in (\ref{4-1}), it is assume that $y_0\in H^1(\Omega)$ with
$y_0=0, x\in\partial\Omega$. By virtue of Theorem \ref{thm3.1}, we can define
the solution map $u\rightarrow y(u)$ of $L^2(Q_0)$ into $W(0,T;V)$.
The solution $y(u)$ is called the state of the control system
(\ref{4-1}). The observation of the state is assumed to be given by
$Gy$. Here $G\in \mathcal {L}(W(0,T;V), S)$ is an operator, which is
called the observer, $S$ is a real Hilbert space of observations.
The cost function associated with the control system (\ref{4-1}) is
given by 
\begin{eqnarray}
\label{4-2}
J(y,u)=\frac12\|Gy-z_d\|_S^2+\frac{\delta}2\|u\|^2_{L^2(Q_0)},
\end{eqnarray}
where $z_d\in S$ is a desired state and $\delta>0$ is fixed.

An optimal control problem about the k-FORQ/MCH equation  is
\begin{equation}
\min J(y,u),
\label{4-3}
\end{equation}
where $(y,u)$ satisfies (\ref{4-1}).

Now, we give Theorem \ref{thm4.1} on the existence of an optimal
solution.

\begin{theorem}
\label{thm4.1} There exists an
optimal control solution $(y^{*},\omega^*)$ to the problem (\ref{4-1}).
\end{theorem}
\begin{proof} Suppose $(y,\omega)$ satisfy the equation
$e(y,\omega)=0$. In view of (\ref{4-2}),  we deduce that
$$
J(y,\omega)\geq\frac{\delta}2\|\omega\|^2_{L^2(Q_0)}.
$$
By Theorem \ref{thm3.1}, we obtain
$$
\|y\|_{W(0,T;V)}\rightarrow\infty~~\hbox{yields}~~\|\omega\|_{L^2(Q_0)}\rightarrow\infty.
$$
Therefore, 
\begin{equation}
\label{4-5}
J(y,\omega)\rightarrow\infty,~~\hbox{when}~~\|(y,\omega)\|_X\rightarrow\infty.
\end{equation}

As the norm is weakly lower semi-continuous, we achieve that $J$ is
weakly lower semi-continuous. Since for all $(y,\omega)\in X$,
$J(y,\omega)\geq 0$,
there exists $\lambda\geq 0$
defined by
$$
\lambda=\inf\{J(y,\omega)|(y,\omega)\in X,~e(y,\omega)=0\},
$$
which means the existence of a minimizing sequence
$\{(y^n,\omega^n)\}_{n\in N}$ in $X$ such that
$$
\lambda=\lim_{n\rightarrow\infty}J(y^n,\omega^n)~~\hbox{and}~~
e(y^n,\omega^n)=0,\quad\forall n\in N.
$$
From (\ref{4-5}), there exists an element $(y^{*},\omega^{*})\in X$ such
that when $n\rightarrow\infty$, 
\begin{align}
\label{4-6}
&y^n\rightarrow y^{*},~\hbox{weakly},~~y\in W(0,T;V),
\\
\label{4-7}
&\omega^n\rightarrow \omega^{*},~\hbox{weakly},~~\omega\in L^2(Q_0).~~~~~
\end{align}
Then, using (\ref{4-6}), we get
\begin{equation}
\label{4-8}
\lim_{n\rightarrow\infty}\int_0^T(y^n_t(x,t)-y^{*}_t,\psi(t))_{V',V}dt=0,\quad\forall\psi\in
L^2(0,T;V).
\end{equation}
Since $W(0, T; V)$ is compactly embedded into $L^2(0, T;
L^{\infty})$, we have $y^n\rightarrow y^*$ strongly in $L^2(0, T;  L^{\infty})$, we can infer that $\|y^n\|_{L^2(0, T; L^\infty)}$ is bounded . As the $y^n\rightarrow y^*$ weakly in $W(0, T; V)$,
$\|y^n\|_{W(0, T; V)}$ is bounded. From embedding theorem, we deduce that $\|y^n\|_{L^2(0, T; L^\infty)}$ is also bounded. As $W(0, T; V)$ is compactly embedded into $C(0, T; H)$ , we then derive $\|y^n\|_{C(0, T; H)}$ is bounded.

On the other hand, we have the following estimates
\begin{align*}
&\|u^n\|_{L^\infty}\leq C\|y^n\|_H,\qquad \|u_x^n\|_{L^\infty}\leq C\|y^n\|_H,
\\
&\|u^*\|_{L^\infty}\leq C\|y^*\|_H,\qquad \|u_x^*\|_{L^\infty}\leq C\|y^*\|_H,
\\
&\|u^n-u^*\|_{L^\infty}\leq\|y^n-y^*\|_{L^\infty},\qquad \|u_x^n-u_x^*\|_{L^\infty}\leq C\|y^n-y^*\|_{L^\infty},
\\
&\|u^n\|_H\leq \|y^n\|_H,\qquad \|u_x^n\|_H\leq C\|y^n\|_H,
\\
&\|u^*\|_H\leq \|y^*\|_H,\qquad \|u_x^*\|_H\leq C\|y^*\|_H.
\end{align*}

Using (\ref{4-7}) again, we derive that 
\begin{eqnarray}
\left|\int_0^T\!\!\!\int_\Omega(B\omega-B\omega^{*})\psi dxdt\right|\rightarrow
0,\quad n\rightarrow\infty,~~\forall \psi\in L^2(0,T;H).\nonumber
\end{eqnarray}

Notice that
\allowdisplaybreaks
\begin{align*}
&\bigg|\int_0^T\int_\Omega\bigg(\big(((u^n)^2-(u^n_x)^2)y^n\big)_x-\big(((u^*)^2-(u^*_x)^2)y^*\big)_x\bigg)\psi \,dx\,dt\bigg|
\\
=&\bigg|\int_0^T\int_\Omega\bigg(\bigg((u^n)^2-(u^n_x)^2\bigg)y^n-\bigg((u^*)^2-(u^*_x)^2\bigg)y^*\bigg)\psi_x \,dx\,dt\bigg|
\\
\leq&\bigg|\int_0^T\int_\Omega\bigg((u^n)^2-(u^*)^2 y^n\bigg)\psi_x \,dx\,dt\bigg|
    +\bigg|\int_0^T\int_\Omega(u^*)^2(y^n-y^*)\psi_x \,dx\,dt\bigg|
\\
&\quad  +\bigg|\int_0^T\int_\Omega\bigg((u_x^n)^2-(u_x^*)^2 y^n\bigg)\psi_x \,dx\,dt\bigg|
    +\bigg|\int_0^T\int_\Omega(u_x^*)^2(y^n-y^*)\psi_x \,dx\,dt\bigg|
\\
=&\,I+II+III+IV
\end{align*}

Now, we deal with $I$
\begin{align*}
I
&=\bigg|\int_0^T\int_\Omega\bigg((u^n)^2-(u^*)^2\bigg)y^n\psi_x \,dx\,dt\bigg|
\\
&\leq\int_0^T\|(u^n)^2-(u^*)^2\|_{L^\infty} \|y^n\|_H\|\psi_x\|_H \,dt
\\
&\leq\|y^n\|_{C(0, T;H)}\|(u^n)^2-(u^*)^2\|_{L^2(0, T; L^\infty)}\|\psi\|_{L^2(0, T; V)}
\\
&\leq\|y^n\|_{C(0, T;H)}\|(y^n)^2-(y^*)^2\|_{L^2(0, T; L^\infty)}\|\psi\|_{L^2(0, T; V)}
\\
&\to 0,\quad \forall\psi\in L^2(0, T;  V).
\end{align*}

and
\begin{align*}
II
&=\bigg|\int_0^T\int_\Omega(u^*)^2(y^n-y^*)\psi_x \,dx\,dt\bigg|
\\
&\leq \int_0^T\|y^n-y^*\|_{L^\infty}\|u^*\|_{L^\infty}\|u^*\|_H\|\psi_x\|_H \,dt
\\
&\leq \int_0^T\|y^n-y^*\|_{L^\infty}\|y^*\|_H\|y^*\|_H\|\psi_x\|_H \,dt
\\
&\leq\|y^n\|^2_{C(0, T;H)}\|(y^n)^2-(y^*)^2\|_{L^2(0, T; L^\infty)}\|\psi\|_{L^2(0, T; V)}
\\
&\to 0,\quad \forall\psi\in L^2(0, T;  V).
\end{align*}
and
\begin{align*}
III
&=\bigg|\int_0^T\int_\Omega\bigg((u_x^n)^2-(u_x^*)^2 y^n\bigg)\psi_x \,dx\,dt\bigg|
\\
&\leq\int_0^T\|u_x^n+u_x^*\|_{L^\infty} \|u_x^n-u_x^*\|_{L^\infty}\|y^n\|_H\|\psi_x\|_H \,dt
\\
&\leq C\int_0^T(\|y^n\|_H+\|y^*\|_H) \|y^n-y^*\|_{L^\infty}\|y^n\|_H\|\psi_x\|_H \,dt
\\
&\leq C\big(\|y^n\|_{C(0, T;H)}+\|y^*\|_{C(0, T;H)}\big)
\\
&\qquad\|y^n\|_{C(0, T;H)}\|(u^n)^2-(u^*)^2\|_{L^2(0, T; L^\infty)}\|\psi\|_{L^2(0, T; V)}
\\
&\to 0,\quad \forall\psi\in L^2(0, T;  V).
\end{align*}
and
\begin{align*}
IV
&=\bigg|\int_0^T\int_\Omega(u^*)^2(y^n-y^*)\psi_x \,dx\,dt\bigg|
\\
&\leq \int_0^T\|y^n-y^*\|_{L^\infty}\|u_x^*\|_{L^\infty}\|u_x^*\|_H\|\psi_x\|_H \,dt
\\
&\leq \int_0^T\|y^n-y^*\|_{L^\infty}\|y^*\|_H\|y^*\|_H\|\psi_x\|_H \,dt
\\
&\leq\|y^n\|^2_{C(0, T;H)}\|(y^n)^2-(y^*)^2\|_{L^2(0, T; L^\infty)}\|\psi\|_{L^2(0, T; V)}
\\
&\to 0,\quad \forall\psi\in L^2(0, T;  V).
\end{align*}

On the other hand,
\begin{align*}
& \bigg|\int_0^T\int_\Omega k(u^n_x-u^*_x)\psi \,dx\,dt\bigg|
\\
=&\bigg|\int_0^T\int_\Omega k(u^n-u^*)\psi_x \,dx\,dt\bigg|
\\
\leq &\int_0^T|k|\|u^n-u^*\|_{L^\infty} \|\psi_x\|_H \,dt
\\
\leq &\int_0^T\|y^n-y^*\|_{L^\infty}\|\psi_x\|_H \,dt
\\
\leq&\|y^n-y^*\|_{L^2(0, T; L^\infty)}\|\psi\|_{L^2(0, T; V)}\to 0,\quad \forall\psi\in L^2(0, T;  V).
\end{align*}

In view of the above discussion, we get
\begin{align*}
e_1(y^{*},u^{*})=0,\quad \forall n\in N.
\end{align*}

As is known, $y^{*}\in W(0,T;V)$, we derive that $y^{*}(0)\in H$.
Since $y^n\rightarrow y^{*}$ weakly in $W(0,T;V)$, we can infer that
$y^n(0)\rightarrow y^{*}(0)$ weakly when $n\rightarrow\infty$. Thus,
we obtain
$$
(y^n(0)-y^{*}(0), \eta)\rightarrow
0,~~n\rightarrow\infty,~~\forall\eta \in H,
$$
which means
$$e_2(y^{*},\omega^{*})=0.
$$ Therefore, we obtain
$$
e(y^{*},\omega^{*})=0,\quad\hbox{in} ~~Y.
$$
So, there exists an optimal solution $(y^{*},u^{*})$ to problem
(\ref{4-1}).

Then, Theorem \ref{thm4.1} is proved.
\end{proof}

\section{First-order nececessary optimality condition}

Let $(y^*, \omega^*)\in X$ be an optimal solution to (\ref{2-3}). If the operator $e'(y^*, \omega^*)$ is surjective then there exist a Lagrange multiplier $\lambda^*\in L^2(V)$ and $\mu^*\in H$ satisfying the first-order necessary optimality condition
$$
L'_c(y^*, \omega^*,\lambda^*,\mu^*)=0,\quad e(y^{*},\omega^{*})=0,\quad\hbox{for} c\geq0.
$$

\begin{lemma}
\label{lem5.1}
The operator $e'(y,\omega)$ is surjective for all $(y,\omega)\in X$, i.e., for all $(g, h)\in Y$ there exists $(m, q)\in X$ such that
\begin{align}
&e_1(y,\omega)(m, q)=g \quad\hbox{in}\quad L^2(V) \label{5-0}
\\
&e_2(y,\omega)(m, q)=h \quad\hbox{in}\quad H \label{5-2}
\end{align}
\end{lemma}

\begin{proof}
By the definition of the operator $e$, we know the Fr$\acute{e}$chet-derivatives of $e$ is
\begin{align*}
&(e'(y,\omega)(m,q),\xi)
\\
=&(m_t, \xi)+\epsilon m_x\xi_x-2uvy\xi_x+2u_xv_xy\xi_x-u^2m\xi_x+u^2_xm\xi_x-kv\xi_x-Bq\xi
\end{align*}
where $y=u-u_{xx}$ and $m=v-v_{xx}$

We define a bilinear form $a(t;\cdot,\cdot): V\times V \rightarrow \mathbb{R}$ for $t\in[0, T] a.e.$ by
$$
a(t;m,\xi)=\int_\Omega(\epsilon m_x\xi_x-2uvy\xi_x+2u_xv_xy\xi_x-u^2m\xi_x+u^2_xm\xi_x-kv\xi_x)dx
$$

Set $q=0$, equation equals to
\begin{equation}
\label{5-3}
 \frac d{dt}(m(t),\xi)+a(t;m,\xi)=(-\triangle g(t),\xi)_{V^*,V}
\end{equation}
for all $\xi\in V$ and $t\in[0, T] a.e.$ and the initial condition
\begin{equation}
\label{5-4}
m(0)=h.
\end{equation}
Next we will prove that there exists a solution to the problem

Applying (\ref{3-16}), we have
\begin{align}
\label{5-5}
&\|u\|_{L^\infty}\leq 2\|y\|_H,\qquad \|u_x\|_{L^\infty}\leq 2\|y\|_H,
\\
\label{5-6}
&\|v\|_{L^\infty}\leq 2\|m\|_H,\qquad\|v_x\|_{L^\infty}\leq 2\|m\|_H.
\end{align}

Thus,using integration by parts and H\"{o}lder's inequality, we have
\begin{align}
\label{5-7}
&  |a(t;m,\xi)|
\nonumber
\\
=&\int_\Omega|\epsilon m_x\xi_x-2uvy\xi_x+2u_xv_xy\xi_x-u^2m\xi_x+u^2_xm\xi_x-kv\xi_x|dx
\nonumber
\\
\leq&\epsilon\|m\|_V\|\xi\|_V+2\|u\|_{L^\infty}\|v\|_{L^\infty}\|y\|_H\|\xi\|_V+2\|u_x\|_{L^\infty}\|v_x\|_{L^\infty}\|y\|_H\|\xi\|_V
\nonumber
\\
&+\|u\|^2_{L^\infty}\|m\|_H\|\xi\|_V+\|u_x\|^2_{L^\infty}\|m\|_H\|\xi\|_V+|k|\|v\|_{L^\infty}\|\xi\|_V
\nonumber
\\
\leq&(\epsilon+|k|)\|m\|_V\|\xi\|_V+4\|y\|^2_H\|m\|_H\|\xi\|_V
\nonumber
\\
\leq&(\epsilon+|k|+C\|y\|^2_{C(H)})\|m\|_V\|\xi\|_V
\end{align}

In the similar way, we obtain that
\begin{align}
\label{5-8}
&  a(t;m,m)
\nonumber
\\
=&\int_\Omega(\epsilon m^2_x-2uvym_x+2u_xv_xym_x-u^2mm_x+u^2_xmm_x-kvm_x)dx
\nonumber
\\
\geq&\epsilon\|m\|^2_V-2\|u\|_{L^\infty}\|v\|_{L^\infty}\|y\|_H\|m\|_V-2\|u_x\|_{L^\infty}\|v_x\|_{L^\infty}\|y\|_H\|m\|_V
\nonumber
\\
&-\|u\|^2_{L^\infty}\|m\|_H\|m_x\|_H-\|u_x\|^2_{L^\infty}\|m\|_H\|m_x\|_H-|k|\|v\|_{L^\infty}\|m_x\|_H
\nonumber
\\
\geq&\epsilon\|m\|^2_V-(6\|y\|^2_H+|k|)\|m\|_H\|m\|_V
\nonumber
\\
\geq&\epsilon\|m\|^2_V-(\frac{48}{\epsilon}\|y\|^4_H\|m\|^2_H+\frac{8|k|}{\epsilon}\|m\|^2_H+\frac{\epsilon}2\|m\|^2_V)
\nonumber
\\
\geq&\frac{\epsilon}2\|m\|^2_V-\frac{24}{\epsilon}\|y\|^4_{C(H)}\|m\|^2_H-\frac{8|k|}{\epsilon}\|m\|^2_H.
\end{align}
Therefore, combing (\ref{5-7}) with (\ref{5-8}),we obtain the existence of an element $m\in(W(V))$ satisfying (\ref{5-0}) and (\ref{5-2})\cite{lions}.
\end{proof}

\begin{theorem}
\label{thm5.1}
Let $(y^*,\omega^*)\in X$ be an optimal solution to (\ref{2-3}), then there exist unique Lagrange multipliers $\lambda^*\in L^2(V)$
and $\mu^*\in H$ satisfying the first-order necessary optimality condition:
\begin{align}
&\frac{\partial L_c}{\partial y}(y^*, \omega^*, \lambda^*, \mu^*)=0\quad\hbox{in}\quad W(V)\label{5-9}
\\
&\frac{\partial L_c}{\partial \omega}(y^*, \omega^*, \lambda^*, \mu^*)=0\quad\hbox{in}\quad L^2(Q_0)\label{5-10}
\end{align}
\end{theorem}

\begin{proof}
Theorem \ref{thm3.1} and Lemma \ref{lem5.1} imply directly that there exist a Lagrange multipliers $\lambda^*\in L^2(V)$
and $\mu^*\in H$ satisfying the first-order necessary optimality condition.

Now, we are going to prove the uniqueness.Recall that $v=(1-\partial_x^2)^{-1}m$ and $u=(1-\partial_x^2)^{-1}y$, then Equations (\ref{5-9}) are equivalent to
\begin{align}
\label{5-11}
0=&(Cy^*-z, Cm)_S+\int_0^T(m_t(t), \lambda^*(t))_{V^*, V}dt+\int_\Omega m(0)\mu^*dx
\nonumber
\\
\qquad&+\int_0^T\int_\Omega(\epsilon m_x\lambda^*_x-2u^*vy^*\lambda^*_x+2u^*_xv_xy^*\lambda^*_x-u^2m\lambda^*_x+kv\lambda^*_x+(u^*)^2_xm\lambda^*_x)dxdt
\nonumber
\\
=&V+VI+VII+VIII
\end{align}
By using intergration by parts, we have
\begin{align}
\label{5-12}
V&=\big(C^\star(Cy^*-z), (1-\partial^2)^{-1}v\big)_{L^2(V^*), L^2(V)}
\nonumber
\\
&=\big((1-\partial^2)C^\star(Cy^*-z), v\big)_{L^2(V^*), L^2(V)}
\end{align}

Furthermore, we rewrite the VI and VII on the right side of \label{5-1} in the following way:
\begin{align}
\label{5-13}
&\qquad VI+VII
\nonumber
\\
&=\int_0^T\frac d{dt}(\lambda^*, m(t))_Hdt+\int_0^T(-\lambda_t^*, m(t))_{V^*, V}dt+(\mu^*, m(0))_H
\nonumber
\\
&=\int_0^T\frac d{dt}(\lambda^*, m(t))_Hdt-(\lambda_t^*, m(t))_{L^2(V^*), L^2(V)}+(\mu^*, m(0))_H
\nonumber
\\
&=\int_0^T\frac d{dt}(\lambda^*, m(t))_Hdt-(\lambda_t^*-\lambda_{txx}^*, v(t))_{L^2(V^*), L^2(V)}dt+(\mu^*, m(0))_H
\nonumber
\\
&=\big(\lambda^*(T), m(T)\big)_H-\big((1-\partial^2_x)\lambda_t^*, v(t)\big)_{L^2(V^*), L^2(V)}+(\mu^*-\lambda^*(0), m(0))_H.
\end{align}

Finally, using integration by parts, we have
\begin{align}
\label{5-14}
VIII=&\bigg((1-\partial^2_x)(-\epsilon\lambda^*_{xx}-(u^*)^2\lambda^*_x+(u^*_x)^2\lambda^*_x)
\nonumber
\\
&\qquad\qquad+2u^*y^*\lambda^*_x-2u^*_{xx}y^*\lambda^*_x-2u^*_xy^*_x\lambda^*_x-2u^*_xy^*\lambda^*_{xx}-k\lambda^*_x, v\bigg)
\end{align}

Form (\ref{5-11})-(\ref{5-14}), we have
\begin{align}
\label{5-15}
0=&\bigg(\lambda^*(T), m(T)\bigg)_H+\bigg(\mu^*-\lambda^*(0), m(0)\bigg)_H
\nonumber
\\
&\quad+\bigg((1-\partial^2_x)(C^\star(Cy^*-z)-\lambda^*_t-\epsilon\lambda^*_{xx}-(u^*)^2\lambda^*_x+(u^*_x)^2\lambda^*_x)
\nonumber
\\
&\qquad+2u^*y^*\lambda^*_x-2u^*_{xx}y^*\lambda^*_x-2u^*_xy^*_x\lambda^*_x-2u^*_xy^*\lambda^*_{xx}-k\lambda^*_x, v(t)\bigg)_{L^2(V^*), L^2(V)}
\end{align}

Choosing $m$ appropriately we claim that
$$
\lambda^*(T)=0\quad\hbox{and}\quad\mu^*=\lambda^*(0).
$$
Thus, we conclude from (\ref{5-15}) that $\lambda^*$ is the weak solution of the backward differential equation
\begin{align}
\label{5-16}
&\lambda^*_t+\epsilon\lambda^*_{xx}
\nonumber
\\
=&C^\star(Cy^*-z)-(u^*)^2\lambda^*_x+(u^*_x)^2\lambda^*_x
\nonumber
\\
&\quad-(1-\partial^2_x)^{-1}(-2u^*y^*\lambda^*_x+2u^*_{xx}y^*\lambda^*_x+2u^*_xy^*_x\lambda^*_x+2u^*_xy^*\lambda^*_{xx}+k\lambda^*_x)
\end{align}
with the intial condition at $t=T$
\begin{equation}
\label{5-17}
\lambda^*(T)=0.
\end{equation}
Let $\tau=T-t$ and $\rho^*(\tau)=\lambda(t)$ for $t\in[0, T]$. Then equation (\ref{5-16}) and (\ref{5-17}) can be transformed into the
forward differential equation
\begin{align}
\label{5-18}
&\rho^*_\tau-\epsilon\rho^*_{xx}
\nonumber
\\
=&-C^\star(Cy^*-z)+(u^*)^2\rho^*_x-(u^*_x)^2\rho^*_x
\nonumber
\\
&\quad+(1-\partial^2_x)^{-1}(-2u^*y^*\rho^*_x+2u^*_{xx}y^*\rho^*_x+2u^*_xy^*_x\rho^*_x+2u^*_xy^*\rho^*_{xx}+k\rho^*_x)
\end{align}
with the intial condition at $t=T$
\begin{equation}
\label{5-19}
\lambda^*(T)=0.
\end{equation}
for all $m=v-v_{xx}\in W(V)$, where $y^*=u^*-u^*_{xx}$. Meanwhile Equation (\ref{5-10}) implies that $\lambda^*=\sigma \omega^*$ in $Q_0$
The existence of a unique $\rho^*$ solving problem (\ref{5-18}) and (\ref{5-19}) is analogously shown as the proof of the existence of a (unique) $v$ solving problem. Consequently, $\lambda^*(t)=\rho(\tau)\in W(V)$ is the unique solution of the problem, and $\mu^*$ is given by $\mu^*=\lambda^*(0)$.
\end{proof}

\section{Second-order sufficient optimality condition}

In this section, we prove the second-order sufficient optimality conditions provided for problem
 (\ref{2-3}) for two choices of the operator $G$ and for the Hilbert space $S$.

First of all, we need to derive an estimate for the lagrange-multiplier $\lambda$.
\begin{lemma}
\label{lem6.1}
The lagrange multiplier $\lambda^*$ satisfies $\lambda^*\in W(V)$ and
\begin{equation}
\label{6-1}
\|\lambda^*\|_{L^2(V)}^2\leq\frac4{3\epsilon}\exp(c_0\tau)\|C^\star(Cy^*-z)\|_{L^2(V^*)}.
\end{equation}
\end{lemma}

\begin{proof}
If $u^*=0$ in (\ref{5-18}), we can obtain that
\begin{equation*}
\label{6-2}
\|\lambda^*\|_{L^2(V)}=\|\rho^*\|_{L^2(V)}\leq \frac1{\epsilon}\|C^\star(Cy^*-z)\|_{L^2(V^*)}.
\end{equation*}

If $u^*\neq0$, then we set $\rho^*(\tau)=\theta^*(\tau)\exp(c_0\tau)$. From (\ref{5-18}) we derive
\begin{align}
\label{6-3}
&\epsilon\|\theta^*\|_{L^2(V)}^2+c_0\|\theta^*\|_{L^2(H)}
\nonumber
\\
\leq& \|C^\star(Cy^*-z)\|_{L^2(V^*)}\|\theta^*\|_{L^2(V)}
\nonumber
\\
+&\int_0^T\int_\Omega|(u^*)^2\theta^*_x\theta^*|dxdt+\int_0^T\int_\Omega|(u_x^*)^2\theta^*_x\theta^*|dxdt
\nonumber
\\
+&\int_0^T\int_\Omega|\big((1-\partial^2_x)^{-1}(-2u^*y^*\theta^*_x+2u^*_{xx}y^*\theta^*_x+2u^*_xy^*_x\theta^*_x
+2u^*_xy^*\theta^*_{xx}+k\theta^*_x)\big)\theta^*|dxdt
\end{align}
Note that
\begin{equation*}
\|u^*\|_{L^\infty}\leq 2\|y^*\|_H,\qquad \|u_x^*\|_{L^\infty}\leq \frac12\|y^*\|_H,
\end{equation*}

By the H\"{o}ler's inequality and Young's inequality, we obtain
\begin{align}
\label{6-5}
\int_0^T\int_\Omega|(u^*)^2\theta^*_x\theta^*|dxdt&\leq \int_0^T\|u^*\|_{L^\infty}^2\|\theta^*\|_V\|\theta^*\|_Hdt
\nonumber
\\
&\leq4\int_0^T\|y^*\|_H^2\|\theta^*\|_V\|\theta^*\|_Hdt
\nonumber
\\
&\leq4\|y^*\|_{C(H)}^2\|\theta^*\|_{L^2(V)}\|\theta^*\|_{L^2(H)}
\nonumber
\\
&\leq\frac{16}{\epsilon}\|y^*\|_{C(H)}^4\|\theta^*\|^2_{L^2(H)}+\frac\epsilon4\|\theta^*\|_{L^2(V)}.
\end{align}
and
\begin{align}
\label{6-6}
\int_0^T\int_\Omega|(u^*_x)^2\theta^*_x\theta^*|dxdt&\leq \int_0^T\|u^*_x\|_{L^\infty}^2\|\theta^*\|_V\|\theta^*\|_Hdt
\nonumber
\\
&\leq\frac14\int_0^T\|y^*\|_H^2\|\theta^*\|_V\|\theta^*\|_Hdt
\nonumber
\\
&\leq\frac14\|y^*\|_{C(H)}^2\|\theta^*\|_{L^2(V)}\|\theta^*\|_{L^2(H)}
\nonumber
\\
&\leq\frac1{16\epsilon}\|y^*\|_{C(H)}^4\|\theta^*\|^2_{L^2(H)}+\frac\epsilon4\|\theta^*\|_{L^2(V)}.
\end{align}
Now, we will deal with the last term of (\ref{6-3})
\begin{align}
\label{6-7}
&\int_0^T\int_\Omega|(1-\partial^2_x)^{-1}(-2u^*y^*\theta^*_x+2u^*_{xx}y^*\theta^*_x+2u^*_xy^*_x\theta^*_x+2u^*_xy^*\theta^*_{xx}+k\theta^*_x)\theta^*|dt
\nonumber
\\
\leq&\int_0^T\|(1-\partial^2_x)^{-1}(-2u^*y^*\theta^*_x+2u^*_{xx}y^*\theta^*_x+2u^*_xy^*_x\theta^*_x+2u^*_xy^*\theta^*_{xx})\|_H\|\theta^*\|_Hdt
\nonumber
\\
=&\int_0^T\|G*(-2u^*y^*\theta^*_x+2u^*_{xx}y^*\theta^*_x+2u^*_xy^*_x\theta^*_x+2u^*_xy^*\theta^*_{xx})\|_H\|\theta^*\|_Hdt
\nonumber
\\
\leq &\int_0^T\|G\|_H\|(-2u^*y^*\theta^*_x+2u^*_{xx}y^*\theta^*_x+2u^*_xy^*_x\theta^*_x+2u^*_xy^*\theta^*_{xx})\|_{L^1}\|\theta^*\|_Hdt
\nonumber
\\
\leq &2\int_0^T\|(-2u^*y^*\theta^*_x+2u^*_{xx}y^*\theta^*_x+2u^*_xy^*_x\theta^*_x+2u^*_xy^*\theta^*_{xx})\|_{L^1}\|\theta^*\|_Hdt
\nonumber
\\
\leq&\int_0^T\big(4\|u^*\|_{L^\infty}\|y^*\|_H\big)\|\theta^*\|_V\|\theta^*\|_Hdt
\nonumber
\\
\leq&\int_0^T\big(4\|y^*\|^2_H\big)\|\theta^*\|_V\|\theta^*\|_Hdt
\nonumber
\\
\leq&\frac{16}\epsilon\|y^*\|^2_H\|\theta^*\|^2_{L^2(H)}+\frac\epsilon4\|\theta^*\|_V
\end{align}
Inserting (\ref{6-5})-(\ref{6-7}) into (\ref{5-11})yields
\begin{align}
\label{6-8}
&\epsilon\|\theta^*\|_{L^2(V)}^2+c_0\|\theta^*\|_{L^2(H)}
\nonumber
\\
\leq&\|C^\star(Cy^*-z)\|_{L^2(V^*)}\|\theta^*\|_{L^2(V)}
+\frac{8+1/16}{\epsilon}\|y^*\|_{C(H)}^4\|\theta^*\|^2_{L^2(H)}+\frac{3\epsilon}4\|\theta^*\|_{L^2(V)}
\end{align}
where $c_0=\frac{8+1/16}{\epsilon}\|y^*\|_{C(H)}^4$, then we obtain
\begin{equation*}
\|\theta^*\|_{L^2(V)}^2\leq\frac4{3\epsilon}\|C^\star(Cy^*-z)\|_{L^2(V^*)}
\end{equation*}

From $\rho^*(\tau)=\lambda^*(t)$ and $\rho^*(\tau)=\theta^*(\tau)\exp(c_0\tau)$, we have
\begin{align}
\label{6-10}
\|\lambda^*\|_{L^2(V)}^2\leq\frac4{3\epsilon}\exp(c_0\tau)\|C^\star(Cy^*-z)\|_{L^2(V^*)}.
\end{align}
From (\ref{6-7}), we derive that
\begin{align}
\label{6-11}
&\|\lambda^*_t\|_{L^2(V^*)}
\nonumber
\\
=&\underset{\|\varphi\|_{L^2(V)=1}}\sup\bigg|\int_0^T\int_\Omega
(-\epsilon\lambda^*_{xx}\varphi+C^\star(Cy^*-z)\varphi-(u^*)^2\lambda^*_x\varphi+(u^*_x)^2\lambda^*_x\varphi
\nonumber
\\
&-((1-\partial^2_x)^{-1}(-2u^*y^*\lambda^*_x+2u^*_{xx}y^*\lambda^*_x+2u^*_xy^*_x\lambda^*_x+2u^*_xy^*\lambda^*_{xx}))\varphi)dxdt\bigg|
\nonumber
\\
\leq&\epsilon\|\lambda\|_{L^2(V)}+\|C^\star(Cy^*-z)\varphi\|_{L^2(V^*)}+(8+\frac14)\|y^*\|_{C(H)}^2\|\lambda\|_{L^2(V)}
\nonumber
\\
=&\|C^\star(Cy^*-z)\varphi\|_{L^2(V^*)}+(\epsilon+(8+\frac14))\|y^*\|_{C(H)}^2\|\lambda\|_{L^2(V)}
\end{align}
Finally, we infer (\ref{6-10}) and (\ref{6-11}) that $\lambda\in W(V)$. We complete the proof of the lemma.
\end{proof}

To obtain the second-order sufficient optimality condition, we need an auxiliary lemma.

\begin{lemma}
\label{lem6.2}
Let $(y, \omega)\in X$ and $(m, q)\in \ker e'(y, \omega)$. Then there exists a constant $c_1>0$ such that
\begin{equation*}
\label{6-12}
\|m\|_{W(V)}^2\leq c_1\|q\|_{L^2(Q_0)}^2.
\end{equation*}
\end{lemma}
\begin{proof}
If $u=0$ then $(m, q)\in \ker e'(y, \omega)$ lead to $m(0)=0$ and
\begin{equation*}
\label{6-13}
(m_t, \xi)+\int_0^T\int_\Omega(\epsilon m_x\xi_x-kv\xi_x-Bq\xi)dxdt=0.
\end{equation*}
for all $\xi\in L^2(V)$. In particular, the choice $\xi=m$ implies that
\begin{equation*}
\label{6-14}
\|m\|_{L^2(V)}\leq \frac1\epsilon\|q\|_{L^2(Q_0)}.
\end{equation*}
where we used
\begin{align*}
-k\int_0^T\int_\Omega vm_xdxdt&=k\int_0^T\int_\Omega (1-\partial^2)^{-1}m_x mdxdt
\\
&=k\int_0^T\int_\Omega m_x (1-\partial^2)mdxdt=0.
\end{align*}
Thus, we conclude
\begin{eqnarray*}
\label{6-15}
\|m_t\|_{L^2(V^*)}&=\underset{\|\xi\|_{L^2(V)=1}}\sup\int_0^T\int_\Omega(\epsilon m_x\xi_x-kv\xi_x-Bq\xi)dxdt
\nonumber
\\
\leq(\epsilon+|k|)\|m\|_{L^2(V)}+\|q\|_{L^2(Q_0)}
\nonumber
\\
\leq(2+ \frac{|k|}\epsilon)\|q\|_{L^2(Q_0)}
\end{eqnarray*}

If $u\neq0$ then $(m, q)\in \ker e'(y, \omega)$ lead to $m(0)=0$ and
\begin{align}
\label{6-16}
(m_t, \xi)+\int_0^T\int_\Omega(\epsilon m_x\xi_x-&2uvy\xi_x+2u_xv_xy\xi_x
\nonumber
\\
-&u^2m\xi_x+u^2_xm\xi_x-kv\xi_x-Bq\xi)dxdt=0
\end{align}
As in the proof of Theorem \ref{thm4.1}, let
\begin{equation}
\label{6-17}
m(\tau)=\gamma(\tau)\exp(c_2\tau)
\end{equation}
from
\begin{equation*}
m(\tau)=v(\tau)-v_{xx}(\tau)
\end{equation*}
we have
\begin{equation*}
\label{6-18}
v(\tau)=\eta(\tau)\exp(c_2\tau)
\end{equation*}
and
\begin{equation*}
\label{6-19}
\eta(\tau)=(1-\partial^2)^{-1}\gamma(\tau)=G*\gamma(\tau)
\end{equation*}
Choose $\xi=m$ in (\ref{6-16}), we obtain
\begin{align*}
\label{6-20}
&c_2\|\gamma\|^2_{L^2(H)}+\epsilon\|\gamma\|^2_{L^2(V)}
\nonumber
\\
\leq&\|q\|_{L^2(Q_0)}\|\gamma\|_{L^2(V)}
\nonumber
\\
\qquad\qquad&+\int_0^T\int_\Omega(-2u\eta y\gamma_x+2u_x\eta_xy\gamma_x-u^2\gamma\gamma_x+u^2_x\gamma\gamma_x-k\eta\gamma_x)dxdt
\nonumber
\\
\leq&\|q\|_{L^2(Q_0)}\|\gamma\|_{L^2(V)}
\nonumber
\\
\qquad\qquad&+\int_0^T(2\|u\|_{L^\infty}\|\eta\|_{L^\infty}\|y\|_H\|\gamma_x\|_H+2\|u_x\|_{L^\infty}\|\eta_x\|_{L^\infty}\|y\|_H\|\gamma_x\|_H
\nonumber
\\
\qquad\qquad&
+\|u\|_{L^\infty}^2\|y\|_H\|\gamma_x\|_H+\|u_x\|_{L^\infty}^2\|y\|_H\|\gamma_x\|_H)dt
\nonumber
\\
\leq&\|q\|_{L^2(Q_0)}\|\gamma\|_{L^2(V)}+6\int_0^T(\|y\|^2_H\|\gamma\|_H\|\gamma\|_V)dt
\nonumber
\\
\leq&\|q\|_{L^2(Q_0)}\|\gamma\|_{L^2(V)}+\frac\epsilon2\|\gamma\|^2_{L^2(V)}+\frac1{12\epsilon}\|y\|^2_{C(H)}\|\gamma\|_{L^2(H)}
\end{align*}
where $c_2=\frac1{12\epsilon}\|y\|^2_{C(H)}$
Then we obtain
\begin{equation}
\label{6-21}
\|\gamma\|_{L^2(V)}\leq \frac2\epsilon\|q\|_{L^2(Q_0)}
\end{equation}
Combining (\ref{6-21}) with (\ref{6-17}) we have
\begin{equation}
\label{6-22}
\|m\|_{L^2(V)}\leq \frac2\epsilon\exp(c_2T)\|q\|_{L^2(Q_0)}
\end{equation}
From (\ref{6-16}) and (\ref{6-22}) we have
\begin{align}
\label{6-23}
&\|m_t\|_{L^2(V^*)}
\nonumber
\\
=&\bigg|\int_0^T\int_\Omega(\epsilon m_x\xi_x-2uvy\xi_x+2u_xv_xy\xi_x
\nonumber
\\
&\qquad\qquad-u^2m\xi_x+u^2_xm\xi_x-kv\xi_x-Bq\xi)dxdt\bigg|+\|q\|_{L^2(Q_0)}
\nonumber
\\
\leq&\bigg(\epsilon+6\|y\|_{C(H)}\bigg)\frac2\epsilon\exp(c_2T)\|q\|_{L^2(Q_0)}+\|q\|_{L^2(Q_0)}
\nonumber
\\
=&\bigg((\epsilon+6\|y\|_{C(H)})\frac2\epsilon\exp(c_2T)+1\bigg)\|q\|_{L^2(Q_0)}.
\end{align}
where we used $\|\xi\|_H\leq\|\xi\|_V$ and $\|m\|_H\leq\|m\|_V$.

From (\ref{6-22}) and (\ref{6-23}), we have
\begin{equation}
\label{6-24}
c_1=\bigg((\epsilon+6\|y\|_{C(H)})\frac2\epsilon\exp(c_2T)+1\bigg)^2+\frac4{\epsilon^2}\exp(2c_2T)
\end{equation}
\end{proof}

\begin{theorem}
Let $c_0, c_1$ and $c_2$ be given in (\ref{6-8}),(\ref{6-24}) and (\ref{6-27}), respectively.

(1)Let $S=W(V), C=I$ and suppose
$$
\|y^*\|_{C(H)}\|y^*-z\|_{L^2(H)}<\frac{3\epsilon}{4C_1}\exp(-c_0T)
$$
then $L''(y^*, \omega^*, \lambda^*, \mu^*)$ is coercive on the whole space $X$.

(2)Let $S=L^2(H)$ and $C$ be the injection of $W(V)$ into $L^2(H)$. If
$$
\|y^*\|_{C(H)}\|C^\star(Cy^*-z)\|_{L^2(H)}<\frac{3\sigma\epsilon}{8CC_1}\exp(-c_0T)
$$
then $L''(y^*, \omega^*, \lambda^*, \mu^*)$ is coercive on the kernel of $e'(y^*, \omega^*)$.
\end{theorem}
\begin{proof}
(1)We observe that for all $(m, q)\in X$
\begin{align}
\label{6-25}
&\bigg(L''(y^*, \omega^*, \lambda^*, \mu^*)(m, q), (m, q)\bigg)_X
\nonumber
\\
=&\|m\|_{W(V)}^2+\sigma\|q\|^2_{L^2(Q_0)}+\int_0^T\int_\Omega b(y^*, m, \lambda^*)dxdt
\end{align}
where $y^*=u^*-u_{xx}^*, m^*=v^*-v_{xx}^*$ and $b(y^*, m, \lambda^*)$ is defined by
\begin{eqnarray*}
\label{6-26}
b(y^*, m, \lambda^*)=-2v^2y^*\lambda_x^*-4u^*vm\lambda_x^*+2v_x^2y^*\lambda_x^*+4u^*_xv_xm\lambda_x^*
\end{eqnarray*}
Then by H\"{o}lder inequality and Young's inequality , we obtain
\begin{align}
\label{6-27}
&\bigg|\int_0^T\int_\Omega b(y^*, m, \lambda^*)dxdt\bigg|
\nonumber
\\
\leq&\int_0^T\bigg(2\|v\|^2_{L^\infty}\|y^*\|_H\|\lambda_x^*\|_H+4\|u^*\|_{L^\infty}\|v\|_{L^\infty}\|m\|_H\|\lambda_x^*\|_H
\nonumber
\\
&\qquad+2\|v_x\|^2_{L^\infty}\|y^*\|_H\|\lambda_x^*\|_H+4\|u_x^*\|_{L^\infty}\|v_x\|_{L^\infty}\|m\|_H\|\lambda_x^*\|_H\bigg)dt
\nonumber
\\
\leq&\int_0^T(9\|m\|^2_H\|y^*\|_H\|\lambda^*\|_V)dt
\nonumber
\\
\leq&9\|\lambda^*\|_{L^2(V)}\|m\|^2_{C(H)}\|y^*\|_{C(H)}
\nonumber
\\
\leq&9 C\|\lambda^*\|_{L^2(V)}\|m\|^2_{W(V)}\|y^*\|_{C(H)}
\end{align}
From (\ref{6-1}), (\ref{6-25}) and (\ref{6-27}), we obtain
\begin{align}
\label{6-28}
&\bigg(L''(y^*, \omega^*, \lambda^*, \mu^*)(m, q), (m, q)\bigg)_X
\nonumber
\\
\geq&\|m\|_{W(V)}^2+\sigma\|q\|_{L^2(Q_0)}^2-9C\|\lambda^*\|_{L^2(V)}\|m\|^2_{W(V)}\|y^*\|_{C(H)}
\nonumber
\\
\geq&\bigg(1-\frac{4C_1}{3\epsilon}\exp(c_0T)\|y^*-z\|_{L^2(H)}\|y^*\|_{C(H)}\bigg)\|m\|_{W(V)}^2+\sigma\|q\|_{L^2(Q_0)}^2
\end{align}
Setting
\begin{equation*}
\kappa=\min\bigg(1-\frac{4C_1}{3\epsilon}\exp(c_0T)\|y^*-z\|_{L^2(H)}\|y^*\|_{C(H)}, \sigma\bigg)>0
\end{equation*}
We derive that
\begin{equation*}
(L''_c(y^*, \omega^*, \lambda^*, \mu^*)(m, q), (m, q))_X\geq \|(m, q)\|_X^2.
\end{equation*}

(2) Using lemma \ref{lem6.1} and \ref{lem6.2} , we obtain that for all $(m, q)\in \ker e'(y^*, \omega^*)$
\begin{align*}
&\bigg(L''(y^*, \omega^*, \lambda^*, \mu^*)(m, q), (m, q)\bigg)_X
\nonumber
\\
\geq&\|Cm\|_{L^2(H)}^2+\sigma\|q\|_{L^2(Q_0)}^2-9C\|\lambda^*\|_{L^2(V)}\|m\|^2_{W(V)}\|y^*\|_{C(H)}
\nonumber
\\
\geq&\sigma\|q\|_{L^2(Q_0)}^2-\frac{4C}{3\epsilon}\exp(c_0T)\|C^\star(Cy^*-z)\|_{L^2(H)}\|y^*\|_{C(H)}\|m\|^2_{W(V)}
\nonumber
\\
\geq&\frac\sigma2\|q\|_{L^2(Q_0)}^2+(\frac\sigma{2c_1}-\frac{4C}{3\epsilon}\exp(c_0T)\|C^\star(Cy^*-z)\|_{L^2(H)}\|y^*\|_{C(H)})\|m\|^2_{W(V)}
\end{align*}
Setting
\begin{equation*}
\kappa=\min\bigg(\frac\sigma{2c_1}-\frac{4C}{3\epsilon}\exp(c_0T)\|C^\star(Cy^*-z)\|_{L^2(H)}\|y^*\|_{C(H)}, \frac\sigma2)>0
\end{equation*}
We derive that
\begin{equation*}
(L''_c(y^*, \omega^*, \lambda^*, \mu^*)(m, q), (m, q))_X\geq \|(m, q)\|_X^2.
\end{equation*}
\end{proof}

\end{document}